\documentclass[12pt]{article}

\usepackage{amsthm, amsmath, amssymb, accents, url}

\newcommand{\utilde}{\undertilde}

\newcommand{\restrict}{\upharpoonright}
\newcommand{\forces}{\Vdash}

\newcommand{\dom}{\mathrm{dom}}

\newcommand{\less}{\mathord{<}}

\newcommand{\DC}{\mathsf{DC}}

\newcommand{\ZF}{\mathsf{ZF}}
\newcommand{\ZFC}{\mathsf{ZFC}}

\newcommand{\bls}{\vspace{\baselineskip}}

\newcommand{\HOD}{\mathrm{HOD}}

\newcommand{\cB}{\mathcal{B}}

\newcommand{\cF}{\mathcal{F}}
\newcommand{\cG}{\mathcal{G}}

\newcommand{\cP}{\mathcal{P}}
\newcommand{\cS}{\mathcal{S}}

\newcommand{\cU}{\mathcal{U}}
\newcommand{\cW}{\mathcal{W}}

\newcommand{\rmc}{\mathrm{c}}

\newcommand{\rmf}{\mathrm{f}}

\newcommand{\bbE}{\mathbb{E}}
\newcommand{\bbN}{\mathbb{N}}
\newcommand{\bbP}{\mathbb{P}}
\newcommand{\bbQ}{\mathbb{Q}}
\newcommand{\bbR}{\mathbb{R}}

\newcommand{\supp}{\mathrm{supp}}

\newcommand{\spn}{\mathrm{span}}

\newcommand{\bP}{\mathbf{P}}

\newcommand{\cl}{\mathrm{cl}}
\newcommand{\cf}{\mathrm{cf}}

\newcommand{\bv}{\mathbf{v}}

\newtheorem{theorem}{Theorem}[section]

\newtheorem{lemma}[theorem]{Lemma}

\newtheorem{definition}[theorem]{Definition}

\newtheorem{claim}[theorem]{Claim}

\newtheorem{remark}[theorem]{Remark}

\newtheorem{question}[theorem]{Question}

\begin{document}
	
	\title{Discontinuous homomorphisms without Hamel bases}
	\author{Paul Larson\thanks{The research of the first author is supported in part by NSF grant DMS-2452139 and BSF grant 2024191. We thank Assaf Shani for comments on an earlier version of this paper.} \and Saharon Shelah\thanks{The research of the second author is supported in part by BSF grant 2024191. Publication 1276 in the second author's publication list.\\MSC2020: 03E25; 03E35; Keywords:Hamel bases, discontinuous homomorphisms, Axiom of Choice, internal direct sums}}


	
	\maketitle
	
	\begin{abstract}
	We produce a model of $\ZF + \DC$ in which there exists a discontinuous homomorphism from $(\bbR, +)$ to itself but no Hamel basis for $\bbR$, and prove a generalization of this result in terms of internal direct sums. 
	\end{abstract}

\section{Introduction} 
We say that a subset of a vector space $S$ is a \emph{Hamel basis} for $S$ if it is a maximal linearly independent subset of $S$ over the scalar field $\bbQ$. A permutation of a Hamel basis for $\bbR$ naturally extends to an automorphism of the group $(\bbR, +)$; if the permutation moves some elements and fixes others the induced homomorphism is discontinuous. In this paper we show (Theorem \ref{mainthrm}) that the existence of a discontinuous homomorphism from $(\bbR, +)$ to itself does not, in $\ZF + \DC$, imply the existence of a Hamel basis for any vector space over $\bbQ$ containing a copy of $(\bbR, +)$. 


Our interest in this problem was inspired by two related results regarding selectors for the Vitali equivalence relation $\bbE_{0}$. In \cite{LZ19} it was shown that the existence of a discontinuous homomorphism from $\bbR$ to itself induces (in $\ZF$) a selector for this equivalence relation (this fact is used in Remark \ref{1dimrem} below). In \cite{GST} it was shown that the existence of a selector for $\bbE_{0}$ does not imply the existence of a Hamel basis (assuming the consistency of a strongly inaccessible cardinal, which can be removed by combining the proof in \cite{GST} with the approach used here). These results naturally lead to the questions answered in this paper. 

In the second part of the paper we consider ways of writing $\bbR$ as an internal direct sum of subspaces. 
The existence of a Hamel basis for $\bbR$ is equivalence to $\bbR$ being an internal direct sum of one-dimensional subspaces (see Remark \ref{1dimrem}). The real line being equal to a nontrivial internal direct sum gives rise to a discontinuous homomorphism. This leaves a range of questions asking when $\bbR$ being an internal direct sum of one type implies that it is an internal direct sum of another type. In Theorem \ref{mainthrm2} we answer one such question by showing that $\bbR$ being a nontrivial internal direct sum does not imply that it is an internal direct sum of countable-dimensional subspaces. 

While the main result of the second part of the paper subsumes the original form of the question answered in the first part of the paper, we include both arguments, in part because the argument from the first part can be modified to produce discontinuous homomorphisms which are not induced by internal direct sums. The nonexistence result proved in the first part is generalized in a different way than the corresponding part of the second result. Many similar generalizations are possible and we leave these to the interested reader. 

The arguments here are similar to earlier ones to due to Horowitz and Shelah \cite{HS21, S85}. They can also be easily adapted to the methods of \cite{GST}. 


\section{Amalgamations}

In this section we prove the key lemma of the paper (Lemma \ref{keylemma}), which will be applied in Sections \ref{1stmodelsec} and \ref{idssec}. We first review some standard facts about amalgamating homomorphisms.

If $(G, +)$ and $(H, +)$ are abelian groups, and $f \colon G \to H$ is a partial function, then we say that $f$ is \emph{additive} (or a \emph{partial homomorphism}) if the equation $f(x + y) = f(x) + f(y)$ holds whenever $x$, $y$ and $x+y$ are in the domain of $f$. 

A group $(G,+)$ is said to be \emph{divisible} if for each $x \in G$ and $n \in \bbN$ there is a $y \in G$ such that $ny = x$ (where $ny$ represents $y$ added to itself $n$ times). The additive group of any vector space over a field containing $\bbQ$ is clearly divisible. We repeatedly and implicitly use the following standard fact. 

\begin{lemma}\label{divisiblelem} 
If $(G, +)$ is a divisible abelian group, $G'$ is a subgroup of $G$ and $h \colon G' \to G$ is a homomorphism, then $h$ extends to a homomorphism from $G$ to itself.  
\end{lemma} 

Suppose that $\pi_{1},\ldots,\pi_{n}$ are partial homomorphisms from one abelian group $(G,+)$ to another abelian group $(H, +)$, such that the domain of each $\pi_{i}$ is a subgroup of $G$. We say that $\pi_{1},\ldots,\pi_{n}$ can be \emph{amalgamated} if there is a partial homomorphism $\pi \colon (G, +) \to (H, +)$ containing each $\pi_{i}$ whose domain is also a subgroup. It is a standard fact (and easy to see) that $\pi_{1},\ldots,\pi_{n}$ can be amalgamated if and only if, for all $a_{i}, b_{i} \in \dom(\pi_{i})$ ($1 \leq i \leq n$) if $a_{1} + \cdots + a_{n} = b_{1}+ \cdots + b_{n}$ then $\pi_{1}(a_{1}) + \cdots + \pi_{n}(a_{n}) = \pi_{1}(b_{1}) + \cdots + \pi_{n}(b_{n})$. Letting $c_{i} = a_{i} - b_{i}$, we get the following equivalent condition: whenever $c_{i} \in \dom(\pi_{i})$ ($1 \leq i \leq n$) are such that $c_{1} + \cdots + c_{n} = 0$, $\pi_{1}(c_{1}) + \cdots + \pi_{n}(c_{n}) = 0$. We state two specific versions of this fact, both of which will be used in the proof of Lemma \ref{keylemma}. 

\begin{lemma}\label{version0} Suppose that $h_{1}$ and $h_{2}$ are partial homomorphisms from an abelian group $(G, +)$ to another abelian group $(H, +)$, and that the domain of each $h_{i}$ is a subgroup of $G$. Then $h_{1}$ and $h_{2}$ can be amalgamated if and only if they agree on the intersection of their domains.  
\end{lemma}

\begin{lemma}\label{version} Suppose that $h_{1}$, $h_{2}$ and $h_{3}$ are partial homomorphisms from an abelian group $(G, +)$ to another abelian group $(H, +)$, and that the domain of each $h_{i}$ is a subgroup of $G$. Then $h_{1}$, $h_{2}$ and $h_{3}$ can be amalgamated if and only if the equation $h_{1}(a_{1}) + h_{2}(a_{2}) = h_{3}(a_{3})$ holds whenever $a_{i}$ ($1 \leq i \leq 3$) are such that $a_{1} + a_{2} = a_{3}$ and each $a_{i}$ is in the domain of the corresponding $h_{i}$. 
\end{lemma} 

It is a classical fact that any Borel (or even Baire-measurable) homomorphism between Polish groups is continuous (see \cite{H, Pettis}). We will use the following variation of this fact. 


\begin{theorem}\label{standard}
If $\epsilon > 0$ and $f \colon (-\epsilon, \epsilon) \to \bbR$ is Borel and additive, then there is a real number $a$ such that $f(x) = ax$ for all $x \in (-\epsilon,\epsilon)$. 
\end{theorem}

\begin{proof}
For each real number $x$, define $g(x)$ to be $nf(x/n)$, for any positive integer $n$ such that $|x/n| < \epsilon/2$. The assumptions imply that $g$ is well-defined, additive and Borel, and that it extends $f$.   
\end{proof}


We will let $P$ denote Cohen forcing, in the form $B(\bbR)/I$, where $B(\bbR)$ denotes the set of Borel subsets of $\bbR$ and $I$ the ideal of meager sets.  It follows then that $P$ is a c.c.c. forcing adding a real number and having continuous reading of names (so every real number in the extension is the result of applying a ground-model Borel function to the generic real), and that, for any Borel set $B \subseteq \bbR^{n+1}$, the set of $\bar{x} \in \bbR^{n}$ for which $(\bar{x}, y) \in B$ for an $I$-positive set of $y$ is also Borel (see Theorem 16.1 of \cite{Kechris}). Since $I$ is shift-invariant, it follows that if $x$ is a $V$-generic real added by $P$, and $y$ is  real number in the ground model, then $x + y$ is also $V$-generic for $P$. We fix the terms $P$ and $I$ for the rest of the paper, and refer the reader to \cite{BJ95} for more information about Cohen forcing and its products.





Real numbers $x_{1}$ and $x_{2}$ are \emph{mutually} $P$-\emph{generic} over $M$ if they are the two $P$-generic reals over $M$ produced by forcing over $M$ with the product forcing $P \times P$. 

The following is the key lemma in the proofs of our main theorems. 

\begin{lemma}\label{keylemma} 
Suppose that $x_{1}$ and $x_{2}$ are mutually $P$-generic reals over $V$, and let $x_{3} = x_{1}-x_{2}$. 
\begin{enumerate}
\item If $y$ is an element of $(\bbR^{V[x_{1}]} + \bbR^{V[x_{2}]}) \cap \bbR^{V[x_{3}]}$, then there exist $a, b \in \bbR^{V}$ such that $y = ax_{3} + b$.  
\item Let
\begin{itemize}
\item $h_{0}$ be a homomorphism from $(\bbR, +)$ to itself in $V$, 
\item $c$ be a real number in $V$, and, 
\item for each $i \in \{1,2,3\}$, $h_{i}$ be, in $V[x_{i}]$, a homomorphism from $(\bbR, +)^{V[x_{i}]}$ to itself extending $h_{0}$, with $h_{i}(dx_{i}) = cdx_{i}$ for all $d \in \bbR^{V}$.
\end{itemize}
Then $h_{1}$, $h_{2}$ and $h_{3}$ can be amalgamated in $V[x_{1}, x_{2}]$.
\end{enumerate}
\end{lemma} 

\begin{proof}
For any $x_{1}$, $x_{2}$, $x_{3}$ and $y$ as given there exist Borel functions $r$, $s$ and $t$ in $V$ such that $y = r(x_{1}) + s(x_{2}) = t(x_{3})$. Then there exists a condition $(B_1, B_2) \in P \times P$ (with $x_{1} \in B_{1}$ and $x_{2} \in B_{2}$) such that, letting $\dot{g}_{1}$ and $\dot{g}_{2}$ be the standard names for the reals produced by each copy of $P$, 
$(B_{1}, B_{2})$ forces that $r(\dot{g}_{1}) + s(\dot{g}_{2}) = t(\dot{g}_{1}-\dot{g}_{2})$. We will show that for any such $B_{1}$, $B_{2}$, $r$, $s$ and $t$, there exist $a, b \in \bbR^{V}$ and a condition $(B'_{1}, B'_{2}) \leq (B_{1}, B_{2})$ forcing that $t(\dot{g}_{1} - \dot{g}_{2})$ is equal to $a(\dot{g}_{1} - \dot{g}_{2}) + b$.

Shrinking $B_{1}$ and $B_{2}$ if necessary, we may assume that there exist real numbers $c_{1}$, $c_{2}$ and $\epsilon > 0$ such that (for each $i \in \{1,2\}$) $B_{i}$ is subset of the interval $(c_i-\epsilon, c_{i}+ \epsilon)$ whose complement in this interval is in $I$. For each $i \in \{1,2\}$ and $j \in \omega$, let $B^{j}_{i} = B_{i} \cap (c_{i}-\epsilon/2^{j}, c_{i} + \epsilon/2^{j})$.

 

 
For each $d \in (-\epsilon/2, \epsilon/2) \cap V$, $(B^{1}_{1}, B^{1}_{2})$ forces that $r(\dot{g}_{1}+ d) + s(\dot{g}_{2} + d)$ is equal to $t((\dot{g}_{1} + d)-(\dot{g}_{2} + d)) = t(\dot{g}_{1} - \dot{g}_{2})$. This implies that $(B^{1}_{1}, B^{1}_{2})$ forces that $r(\dot{g}_{1} + d) + s(\dot{g}_{2} + d)$ is the same for all values of $d \in (-\epsilon/2, \epsilon/2) \cap V$. It follows then that $(B^{1}_{1}, B^{1}_{2})$ forces that $r(\dot{g}_1 + d) - r(\dot{g}_{1}) = s(\dot{g}_{2}) - s(\dot{g}_{2}+d)$ for all such values of $d$. Since the realizations of $\dot{g}_{1}$ and $\dot{g}_{2}$ will be mutually generic, this value must be in $V$, and must depend only on $d$. 
Let $u(d)$ be the corresponding value as forced by $B'_{1}$ (for $r$) and $B'_{2}$ (for $s$). Then $u$ is also a Borel function, since $u(d)$ is the unique value taken by the Borel function $r(x + d) - r(x)$ on an $I$-large subset of $B^{1}_{1}$ (here we use the fact that the ``comeagerly many" quantifier preserves Borelness).



We claim that $u(d + e) = u(d) + u(e)$ for all $d, e \in (-\epsilon/4, \epsilon/4)$. To see this, note that $B^{1}_{1}$ forces that $r(\dot{g}_{1} + (d + e)) = r(\dot{g}_{1}) + u(d + e)$, and $B^{2}_{1}$ forces that \[r(\dot{g}_{1} + (d + e)) = r((\dot{g}_{1} + d) + e) = r(\dot{g}_{1} + d) + u(e) = r(\dot{g}_{1}) + u(d) + u(e).\] 
Since the restriction of $u$ to $(-\epsilon/4, \epsilon/4)$ is an additive Borel function, it follows from Lemma \ref{standard} that it is multiplication by some real number $a$. 

It follows that there exists a real number $b_{1} \in V$ such that $B^{3}_{1}$ forces that $r(\dot{g}_{1}) - a\dot{g}_{1} = b_{1}$. To see this, suppose that two conditions $C, D \leq B^{3}_{1}$ forced the value of $r(\dot{g}_{1}) - a\dot{g}_{1}$ to be, respectively, less and more than some rational number $q$. We may assume that $D = C + d$, for some $d \in (-\epsilon/4,\epsilon/4)$.  
Then $C$ forces that \[q <  r(\dot{g}_{1} + d) - a(\dot{g}_{1} + d) = r(\dot{g}_{1}) + ad - a\dot{g}_{1} - ad < q.\]
An analogous argument shows that $B^{3}_{2}$ forces that $s(\dot{g}_{2}) + a\dot{g}_{2} = b_{2}$, for some $b_{2} \in \bbR$. It follows that $(B^{3}_{1}, B^{3}_{2})$ forces that \[t(\dot{g}_{1} - \dot{g}_{2}) = r(\dot{g}_{1}) + s(\dot{g}_{2}) = a(\dot{g}_{1} - \dot{g}_{2}) + (b_{1}+b_{2}).\] 

For the second part, 
by Lemma \ref{version} above, it suffices to show that \[h_1(y_1) + h_2(y_2) = h_3(y_3)\] holds whenever $y_{i} \in \bbR^{V[x_{i}]}$ $(1 \leq i \leq 3)$ are such that $y_1 + y_2 = y_3$.
Fixing such $y_{1}$, $y_{2}$ and $y_{3}$, we have from the first part of the lemma that $y_{3} = ax_{3} + b$, for some $a,b \in \bbR^{V}$. Then 
\begin{eqnarray*}
h_{3}(y_3) &=& h_{3}(ax_{3} + b)\\
&=& h_{3}(ax_{3}) + h_{3}(b)\\
&=& acx_{3} + h_{0}(b)\\
&=& ac(x_{1}-x_{2}) + h_{0}(b)\\\
&=& acx_{1} + h_{1}(b) - acx_{2}\\
&=& h_{1}(ax_{1} + b) - h_{2}(ax_{2}).
\end{eqnarray*}

We also have that $y_{1} + y_{2} = y_{3} = ax_{3} + b = a(x_{1}-x_{2}) + b = (ax_{1} + b) - ax_{2}$. Since $h_{1}$ and $h_{2}$ are extensions of $h_{0}$ existing in mutually generic extensions of $V$ (so the intersection of their domains is $\bbR^{V}$) they can be amalgamated, which implies that \[h_{1}(y_1) + h_2(y_2) = h_{1}(ax + b) - h_{2}(ax).\] 
\end{proof}

\section{The first model}\label{1stmodelsec}



Let $\kappa$ be a cardinal of uncountable cofinality such that $\kappa^{< \kappa} = \kappa$. For each $d \subseteq \kappa$ let $P_{d}$ be the finite support product of our fixed partial order $P$, indexed by the elements of $d$. So a condition $p$ in $P_{d}$ has the form \[\{ B^{p}_{\alpha} : \alpha \in \supp(p)\},\] where $\supp(p)$ is a finite subset of $d$ and each $B^{p}_{\alpha}$ is a Borel $I$-positive subset of $\bbR$. Then $p_{2} \leq p_{1}$ in $P_{d}$ if and only if $\supp(p_{1}) \subseteq \supp(p_{2})$ and $B^{p_{2}}_{\alpha} \setminus B^{p_{1}}_{\alpha} \in I$ for each $\alpha \in \supp(p_{1})$.
We write $G_{d}$ for a generic filter for $P_{d}$, and $\dot{g}_{\alpha}$ for the natural $P_{\{\alpha\}}$-name for the generic real added by $P_{\{\alpha\}}$.

Let $Q$ be the following partial order. Conditions are pairs $(d, \dot{h})$ such that $d$ is a countable subset of $\kappa$ and $\dot{h}$ is a $P_{d}$-name for a homomorphism from $(\bbR,+)$ to itself. We allow $d = \emptyset$, in which case $\dot{h}$ is a check-name for a homomorphism in the ground model. The order is: 
$(d_1, \dot{h}_1) \leq (d_0, \dot{h}_0)$ if $d_0 \subseteq d_1$ and $1_{P_{d_1}}$ forces that $\dot{h}_{0,G_{d_0}} \subseteq \dot{h}_{1,G_{d_1}}$. Given $q \in Q$ we write $d_{q}$ and $\dot{h}_{q}$ for the first and second coordinates of $q$, respectively.

We adopt the convention that undertilded names (like $\utilde{B}$) are $Q$-names and dotted names (like $\dot{h}$) are $P_{d}$-names (for some $d \subseteq \kappa$), often existing in the ground model $V$. Combining the two notations, we get things like $\utilde{\dot{H}}$, the natural $Q$-name for the $P_{\kappa}$-name $\bigcup\{ \dot{h} : (d,\dot{h}) \in K\}$, where $K$ denotes the $Q$-generic filter.



The two following lemmas imply that $\utilde{\dot{H}}_{K}$ is, the in the $Q$-extension $V[K]$, a $P_{\kappa}$-name for a homomorphism from $(\bbR^{V[K, G_{\kappa}]}, +)$ to itself. Both lemmas follow from the remarks in the previous section, specifically Lemma \ref{divisiblelem} for $\bbR$, which says that any partial homomorphism from a subgroup of $(\bbR, +)$ to $(\bbR, +)$ can be extended to a total homomorphism (in the ground model and in any forcing extension). The genericity of $K$ will guarantee that the realization of $\utilde{\dot{H}}_{K}$ is discontinuous. 

\begin{lemma}\label{ohone} For each $\alpha \in \kappa$, the set of $q \in Q$ with $\alpha \in d_{q}$ is dense. 
\end{lemma}

\begin{lemma}\label{ohtwo} Every descending $\omega$-sequence in $Q$ has a lower bound. 
\end{lemma} 

We note several consequences of Lemmas \ref{ohone} and \ref{ohtwo}.
\begin{itemize}
\item The partial order $P_{\kappa}$ is the same in $V$ and in any $Q$-extension $V[K]$, from which it follows that the forcing iteration $Q * P_{\kappa}$ and the product forcing $Q \times P_{\kappa}$ are forcing-equivalent. 

\item Densely many conditions $(q, p) \in Q \times P_{\kappa}$ have the property that $\supp(p) \subseteq d_{q}$. We will call such conditions \emph{normal}.

\item Letting $(K, G_{\kappa})$ denote a $V$-generic filter for $Q*P_{\kappa}$, every every element of $\bbR^{V[K, G_{\kappa}]}$ is an element of $\bbR^{V[G_{d}]}$, for some countable $d \subseteq \kappa$. In particular, $\bbR^{V[K, G_{\kappa}]} = \bbR^{V[G_{\kappa}]}$. 
\end{itemize} 

Theorem \ref{mainthrm} below shows that forcing with $Q * P_{\kappa}$ produces a model with a discontinuous homomorphism from $\bbR$ to $\bbR$ and no Hamel basis for $\bbR$. We strengthen the nonexistence part of the theorem by replacing $\bbR$ with an arbitrary Polish vector space containing a copy of $\bbR$. Most of the rest of this section is a proof of the theorem. 



\begin{theorem}\label{mainthrm} 
Suppose that $(K, G_{\kappa})$ is a $V$-generic filter for $Q * P_{\kappa}$, and let $W$ be the model \[\HOD_{V, \bbR^{V[G_{\kappa}]}, \utilde{\dot{H}}_{K, G_{\kappa}}}\] as defined in $V[K, G_{\kappa}]$. Then $\utilde{\dot{H}}_{K, G_{\kappa}}$ is a discontinuous homomorphism from $(\bbR, +)^{V[K, G_{\kappa}]}$ to itself, and $\DC$ holds in $W$. Furthermore, if, in $V$, $(S, +_{S})$ is a Polish vector space over $\bbQ$ in for which there exists a continuous injective homomorphism in $V$ from $(\bbR, +)$ to $(S, +_{S})$, then   
$W$ does not contain a Hamel basis for $(S, +_{S})$. 
\end{theorem} 



That $W$ satisfies $\DC$ follows from standard arguments, using the fact that $W$
is an inner model of $V[K, G_{\kappa}]$ (a model of Choice), that it contains the reals of this model, and is closed under ordinal definability. Briefly, the point is that, given a tree $T$ in $W$ without terminal nodes, there exist a set $v$ in $V$ (a wellordering of a sufficiently large set, for instance), a sequence of real numbers $\langle x_{n} : n \in \omega\rangle$ and a branch $b$ through $T$ such that, for each $n \in \omega$, the $n$th element of $b$ is definable from $x_{n}$, $v$ and $\utilde{\dot{H}}_{K, G_{\kappa}}$.   


The following amalgamation lemma follows from Lemma \ref{version0} and the fact that mutually generic extensions have no new real numbers in common.  


\begin{lemma}\label{ohthree} Suppose that $(q_{1}, p_{1})$ and $(q_{2}, p_{2})$ are conditions in $Q \times P_{\kappa}$. Suppose that $q \in Q$ is weaker than both $q_{1}$ and $q_{2}$, with $d_{q} = d_{q_{1}} \cap d_{q_{2}}$, and 
that $B^{p_{1}}_{\alpha} \cap B^{p_{2}}_{\alpha} \in I^{+}$ for all $\alpha \in \supp(p_{1}) \cap \supp(p_{2})$. Then $(q_{1}, p_{1})$ and $(q_{2}, p_{2})$ are compatible. 
\end{lemma}

We note that the indices of the conditions in $P_{\kappa}$ play no role in the creation of the model $W$, and that any partial injection from $\kappa$ to $\kappa$ maps conditions to isomorphic conditions. 

\begin{remark}\label{autorem} If $d$ is a subset of $\kappa$, and $f \colon d \to \kappa$ is injective, then
$f$ induces a function $f_{Q}$ on the set of $q \in Q$ with $d_{q} \subseteq d$ and an isomorphism $f_{P} \colon P_{d} \to P_{f[d]}$ such that for each $q \in \dom(f_{Q})$ and $p \in P_{d}$, 
\begin{itemize}
\item $\supp(f_{P}(p)) = f[\supp(p)]$;
\item $B^{f_{P}(p)}_{f(\alpha)} = B^{p}_{\alpha}$ for all $\alpha \in \supp(p)$;
\item $d_{f_{Q}(q)} = f[d_{q}]$;
\item for any $V$-generic filter $G \subseteq P_{q_{d}}$, $f_{P}[G]$ is a $V$-generic filter for $P_{f[d_{q}]}$, and $\dot{h}_{q, G} = \dot{h}_{f_{Q}(q), f_{P}[G]}$. 
\end{itemize}
In the case where $d = \kappa$, we get that, if $(K, G_{\kappa})$ is $V$-generic for $Q*P_{\kappa}$, then $(f_{Q}[K],f_{P}[G_{\kappa}])$ is also $V$-generic, and the two generic extensions give rise to the same model $W$. 
\end{remark}





Suppose toward a contradiction that $\utilde{\dot{B}}$ is a $Q$-name for a $P_{\kappa}$-name for a Hamel basis for $(S, +_{S})$ in $\HOD_{V, \bbR^{V[G_{\kappa}]}, \utilde{\dot{H}}_{K, G_{\kappa}}}$, as forced by some $Q\times P_{\kappa}$-condition $(q_{0}, p_{0})$. 
Replacing $(q_{0}, p_{0})$ with a stronger condition if necessary, we may assume that it is normal, and that there exist
a $v \in V$, a formula $\varphi$, a cardinal $\lambda$ and $P_{d_{q_{0}}}$-names $\dot{r}_{1},\ldots,\dot{r}_{k}$ for elements of $\bbR$ such that 
\[(q_{0},p_{0}) \forces_{Q \times P_{\kappa}} \utilde{\dot{B}} = \{ y \in S : V_{\lambda}[K,G_{\kappa}] \models \varphi(y, v, \utilde{\dot{H}}_{K, G_{\kappa}}, \dot{r}_{1,G_{d_{q_{0}}}},\ldots,\dot{r}_{k,G_{d_{q_{0}}}})\}.\]

Below we write $\tau \in \utilde{\dot{B}}$ as an abbreviation for  the statement \[V_{\lambda}[K,G_{\kappa}] \models \varphi(\tau_{G_{\kappa}}, v, \utilde{\dot{H}}_{K, G_{\kappa}}, \dot{r}_{1,G_{d_{q_{0}}}},\ldots,\dot{r}_{k,G_{d_{q_{0}}}}).\]

The first key point is that, for each $d \subseteq \kappa$, the restriction of $\utilde{\dot{B}}_{K, G_{\kappa}}$ to $S^{V[G_{d}]}$ is decided by the restrictions of $K$ and $G_{\kappa}$ to $d$. 

\begin{lemma}\label{locallem1} 
For each normal $(q,p) \leq (q_{0}, p_{0})$, for each $P_{d_{q}}$-name $\tau$ for an element of $S$, the statement $\tau \in \utilde{\dot{B}}$ is decided by $(q,p')$ for densely many $p' \in P_{d_{q}}$ below $p$. Moreover, for each such $\tau$ there is a $P_{d_{q}}$-name $\sigma$ for a finite subset of $\utilde{\dot{B}}$ such that $p_{0}$ forces that the realization of $\tau$ will be a linear combination of the members of $\sigma$.  
\end{lemma}  

\begin{proof} 
If the first part of lemma failed for some $(q, p)$ and $\tau$ we could find conditions $(q_1,p_1), (q_2,p_2) \leq (q,p)$ forcing opposite truth values for the statement $\tau \in \utilde{\dot{B}}$, with $p_{1} \restrict q_{d} = p_{2} \restrict q_{d}$. 
Applying Remark \ref{autorem} to an injection $f \colon d_{q_{2}} \to \kappa$ fixing $d_{q}$ pointwise and mapping $d_{q_{2}} \setminus d_{q}$ to a set disjoint from $d_{q_{1}} \setminus d_{q}$, we may assume in addition that
$d_{q_{1}} \cap d_{q_{2}} = d_{q}$.
This is impossible since  by Lemma \ref{ohthree}, $(q_{1}, p_{1})$ and $(q_{2}, p_{2})$ would also be compatible. 

The second part of the lemma is proved similarly, by letting $\sigma$ be any name for the unique finite subset of $\utilde{\dot{B}}$ of which the realization of $\tau$ will be a linear combination. An argument like the one just given shows that $G_{d_{q}}$ must decide the size of $\sigma$ and the members of $\sigma$, 
using (for instance) the fact that any finite set $S$ is coded by the set of basic open subsets it intersects. 
\end{proof} 


We can then let, for each $q \leq q_{0}$, $\dot{B}^{q}$ be the $P_{d_{q}}$-name for the value of $\utilde{\dot{B}}_{K, G_{\kappa}} \cap S^{V[G_{d_{q}}]}$, for any $Q\times P_{\kappa}$-generic filter $(K, G_{\kappa})$ containing $(q_{0}, p_{0})$. Then $p_{0}$ forces in $P_{d_{q}}$ that the realization of $\dot{B}^{q}$ will be a maximal linearly independent subset of the space $S$ as interpreted in the $P_{d_{q}}$-extension. 

We now show that the realization of $\dot{B}^{q}$ depends only on the realization of $\dot{h}_{q}$, which gives us one half of our desired contradiction. 

\begin{lemma}\label{rephraselem} Let $d$ be a countable subset of $\kappa$, and suppose that $q_{1}$ and $q_{2}$ are conditions below $q_{0}$ with $d_{q_{1}} = d_{q_{2}} = d$. 
If $G^{1}_{d}$
and $G^{2}_{d}$ are $V$-generic filters for 
$P_{d}$ such that 
\begin{itemize}
\item $p_{0} \in G^{1}_{d} \cap G^{2}_{d}$, 
\item $V[G^{1}_{d}] = V[G^{2}_{d}]$ and 
\item $\dot{h}_{q_{1}, G^{1}_{d}} = \dot{h}_{q_{2}, G^{2}_{d}}$,
\end{itemize}
then $\dot{B}^{q_{1}}_{G^{1}_{d}} = \dot{B}^{q_{2}}_{G^{2}_{d}}$. 
\end{lemma} 

\begin{proof}
Let $G^{*}$ be $V[G^{1}_{d}]$-generic for $P_{\kappa \setminus d}$. Then $(G^{1}_{d}, G^{*})$ and $(G^{2}_{d}, G^{*})$ naturally induce $V$-generic filters $G_{1}$ and $G_{2}$ for $P_{\kappa}$, and $V[G_{1}] = V[G_{2}]$. It suffices to see that there exist $V[G_{1}]$-generic filters $K_{1}$ and $K_{2}$ for $Q^{V}$, with $q_{1} \in K_{1}$ and $q_{2} \in K_{2}$ such that the models $W$ induced by $(K_{1}, G_{1})$ and $(K_{2}, G_{2})$ are the same. We can do this in a forcing extension of $V[G_{1}]$ in which $\cP(Q)^{V}$ is countable. Say that a pair $(q'_{1}, q'_{2})\in Q^{V} \times Q^{V}$ is \emph{good} if $q'_{1} \leq q_{1}$, $q'_{2} \leq q_{2}$, $d_{q'_{1}} = d_{q'_{2}}$ and $\dot{h}_{q'_{1},G_{1} \restrict d_{q'_{1}}} = \dot{h}_{q'_{2}, G_{2} \restrict d_{q'_{2}}}$. Then $(q_{1}, q_{2})$ is good, and it suffices to show that whenever $(q'_{1}, q'_{2})$ is good and $q''_{1} \leq q'_{1}$, there exists a $q''_{2} \leq q'_{2}$ such that $(q''_{1}, q''_{2})$ is good (since the corresponding fact with $1$ and $2$ reversed will follow by symmetry). 

Given such $(q'_{1}, q'_{2})$ and $q''_{1}$, let $d_{q''_{2}}$ be $d_{q''_{1}}$, and call this set $d''$. Since $V[G^{1}_{d}] = V[G^{2}_{d}]$, \[V[G_{1} \restrict d''] = V[G_{2} \restrict d''],\] so
there exists a $P_{d''}$-name $\dot{h}$ such that $\dot{h}_{G_{2} \restrict d''} = \dot{h}_{q''_{1},G_{1}\restrict d''}$. Since $\dot{h}_{q''_{1},G_{1}\restrict d''}$ extends $\dot{h}_{q'_{2}, G_{2} \restrict d_{q'_{2}}}$, we may assume that every condition in $P_{d''}$ forces that $\dot{h}$ extends $\dot{h}_{q'_{2}}$, and we can let $\dot{h}_{q''_{2}}$ be this $\dot{h}$. 
\end{proof}

We now derive the other half of the contradiction.
Fix two ordinals, $\alpha_{1} < \alpha_{2}$ in $\kappa \setminus d_{q_{0}}$ and let 
\begin{itemize}
\item $d_{0} = d_{q_{0}}$, 
\item $d_{1} = d_{0} \cup \{\alpha_{1}\}$, 
\item $d_{2} = d_{0} \cup \{\alpha_{2}\}$ and 
\item 
$d = d_{q_{0}} \cup \{\alpha_{1}, \alpha_{2}\}$.
\end{itemize}  
We will find generic filters $G_{d}$ and $G'_{d}$ for $P_{d}$, such that
\begin{enumerate}
\item $p_{0}\in G_{d} \restrict P_{d_{0}} = G'_{d} \restrict P_{d_{0}}$,
\item $G_{d} \restrict P_{\{\alpha_{1}\}} = G'_{d} \restrict P_{\{\alpha_{1}\}}$ and
\item $V[G_{d}] = V[G'_{d}]$, 
\end{enumerate}
and a homomorphism $h \in V[G_{d}]$ from $(\bbR^{V[G_{d}]}, +)$ to itself such that
\begin{itemize}
\item $h$ extends  $\dot{h}_{q_{0},G_{d} \restrict P_{d_{0}}}$,
\item $h \restrict V[G_{d} \restrict P_{d_{1}}] \in  V[G_{d} \restrict P_{d_{1}}]$,
\item $h \restrict V[G_{d} \restrict P_{d_{2}}] \in  V[G_{d} \restrict P_{d_{2}}]$ and
\item $h \restrict V[G'_{d} \restrict P_{d_{2}}] \in  V[G'_{d} \restrict P_{d_{2}}]$,
\end{itemize}
in such a way that no Hamel basis B for $S^{V[G_{d}]}$ can have the property that 
\begin{itemize}
\item $B \cap V[G_{d} \restrict P_{d_{1}}]$ is a Hamel basis for $S$ in $V[G_{d} \restrict P_{d_{1}}]$,
\item $B \cap V[G_{d} \restrict P_{d_{2}}]$ is a Hamel basis for $S$ in $V[G_{d} \restrict P_{d_{2}}]$ and
\item $B \cap V[G'_{d} \restrict P_{d_{2}}]$ is a Hamel basis for $S$ in $V[G'_{d} \restrict P_{d_{2}}]$.
\end{itemize}
There will exist then $P_{d}$-names $\dot{h}$, $\dot{h}'$, $\dot{h}_{1}$, $\dot{h}_{2}$ and $\dot{h}'_{2}$ such that
\begin{itemize}
\item $(d, \dot{h})$, $(d, \dot{h}')$, $(d_1, \dot{h}_{1})$, $(d_{2}, \dot{h}_{2})$ and $(d_{2}, \dot{h}'_{2})$ are conditions in $Q$, 
\item $(d, \dot{h}),(d, \dot{h}') \leq (d_{1}, \dot{h}_{1}) \leq q_{0}$, 
\item $(d, \dot{h}) \leq (d_{2}, \dot{h}_{2}) \leq q_{0}$, 
\item $(d, \dot{h}') \leq (d_{2}, \dot{h}'_{2})\leq q_{0}$ and 
\item $\dot{h}_{G_{d}} = \dot{h}'_{G'_{d}} = h$. 
\end{itemize} 
This will finish the proof, by contradicting Lemma \ref{rephraselem}.



Let $G_{d_{q_{0}}}$ be a $V$-generic filter for $P_{d_{q_{0}}}$ containing $p_{0}$. 
Let $x_{1}$ and $x_{2}$ be mutually $P$-generic reals over $V[G_{d_{q_{0}}}]$, and let $x_{3} = x_{1}-x_{2}$. Let $G_{d}$ be the $P_{d}$-generic filter extending $G_{d_{q_{0}}}$ giving rise to $x_{1}$ in coordinate $\alpha_{1}$ and $x_{2}$ in coordinate $\alpha_{2}$.  Let $G'_{d}$ be the $P_{d}$-generic filter extending $G_{d_{q_{0}}}$ giving rise to $x_{1}$ in coordinate $\alpha_{1}$ and $x_{3}$ in coordinate $\alpha_{2}$. Then $G_{d}$ and $G'_{d}$ satisfy conditions (1)-(3) above. 

Fix a real number $c \in V[G_{d_{q_{0}}}]$, and for each $i \in \{1,2,3\}$ let $h_{i}$ in $V[G_{d_{q_{0}}}][x_{i}]$ be a homomorphism from $(\bbR,+)^{V[G_{d_{q_{0}}}][x_{i}]}$ to itself extending  the homomorphism $\dot{h}_{q_{0}, G_{d_{q_{0}}}}$, with $h_{i}(dx_{i}) = cdx_{i}$ for all $d \in \bbR^{V[G_{d_{q_{0}}}]}$. 
Applying Lemma \ref{keylemma}, let $h$ be an amalgamation of $h_{1}$, $h_{2}$ and $h_{3}$ in $V[G_{d}]$. 

If $\dot{B}^{q}_{G_{d_{q}}} = \dot{B}^{q'}_{G_{d_{q'}}}$, then the restriction of this set to each of the models $V[G_{d_{q_{0}}}][x_{1}]$, $V[G_{d_{q_{0}}}][x_{2}]$ and $V[G_{d_{q_{0}}}][x_{3}]$ would be a Hamel basis in the corresponding model for the corresponding version of $S$.
The remaining point is that there cannot be a Hamel basis $B$ for $(S, +_{S})$ with this property. 
To see this, note first that since $x_{1}$, $x_{2}$ and $x_{3}$ are pairwise mutually generic over $V[G_{d_{q_{0}}}]$, the intersection of any two of these models is $V[G_{d_{q_{0}}}]$. Now let $\pi$ be (in $V$), a continuous injective homomorphism from $(\bbR, +)$ to $(S, +_{S})$. 
If a $B$ as above did exist, then each of $\pi(x_{1})$, $\pi(x_{2})$ and $\pi(x_{3})$ would be a linear combination of elements of the corresponding set $B \cap V[G_{d_{q_{0}}}][x_{i}]$ (with rational coefficients), and a linear combination of elements of $B \cap V[G_{d}]$ in a unique way. The equation $\pi(x_{3}) +_{S} \pi(x_{2}) = \pi(x_{1})$ however then gives two different linear combinations for $\pi(x_{1})$. This finishes the proof of Theorem \ref{mainthrm}.



\begin{remark}
The generic homomorphism given by the construction above is surjective but not injective. The approach above leads to several possible variations, including the following. To prove the corresponding versions of Lemma \ref{divisiblelem}, \ref{version0} and \ref{version}, and apply the second part of Lemma \ref{keylemma} with the appropriate value of $c$ if necessary. The corresponding version of $Q$ is the natural one, except possibly for the first case below.  
\begin{enumerate}
\item Requiring $h$ to be injective but not surjective. This seems to be the most subtle of the five variations listed here. Naturally, one should pick a real to keep out of the range of the generic homomorphism. However, letting $\dot{h}$ be any name for an injective homomorphism missing this real will not work. Instead, one can naturally associate to each real $x$ in the $P_{\kappa}$-extension the minimal $a \subseteq \kappa$ for which $x \in V[G_{a}]$, and require that $h(x)$ has the same minimal $a$. Additional conditions should be added to implement Lemma \ref{keylemma}. We leave the details to the interested reader.  
\item Requiring $h$ to be bijective. Here it is natural to require each $\dot{h}$ to be a name for an automorphism of $(\bbR, +)$. 
\item Requiring the range of $h$ to be $(\bbQ, +)$. Here the adaptation is straightforward, letting $c$ be $0$ in the application of Lemma \ref{keylemma}. 
\item Requiring that $h(h(x)) = x$ for all $x \in \bbR$. Here again the adaptation is straightforward, letting $c$ be $1$ in the application of Lemma \ref{keylemma}. 
\item Requiring $\bbR/\{ x : h(x) = x\}$ to have dimension 2. Here is is natural to start with an an automorphism $h$ in the ground model derived from transposition of a Hamel basis, and work below a condition of the form $(\emptyset, \check{h})$. Again, let $c = 1$. 
\end{enumerate}
\end{remark}
   
An interesting question (asked of us by Asaf Karagila) is whether the existence of discontinuous homomorphisms of the types listed above can be separated. A homomorphism of the fifth type easily gives one of the fourth type, but otherwise these questions are open as far as we know. The following related question also appears to be open.


\begin{question}[Schilhan, Zapletal]\label{SZq} Does the existence of a Hamel basis for $\bbR^{2}$ imply the existence of a Hamel basis for $\bbR$? 
\end{question}

We note that if $f \colon \bbR^{n} \to \bbR^{m}$ is a discontinuous homomorphism, then for some $i,j$ the function $x \mapsto f_{i}(x\mathbf{e}_{j})$ is a discontinuous homomorphism from $\bbR$ to $\bbR$, where $f_{1},\ldots,f_{m}$ are the component functions of $f$ and $\mathbf{e}_{1},\ldots,\mathbf{e}_{n}$ are the standard basis vectors for $\mathbb{R}^{n}$. It follows that a negative answer to Question \ref{SZq} would give another proof of the main result of this paper. 

\section{Internal direct sums}\label{idssec}

We say that collection of subspaces $\cW = \{ W_{i} : i \in I\}$ of a vector space $S$ is \emph{linearly independent} if the trivial subspace $\{0\}$ is not in $\cW$ and $0$ cannot be written as a sum of nonzero elements of a nonempty set of distinct $W_{i}$'s. If in addition every element of $S$ can be written as a sum of elements of the $W_{i}$'s (i.e., if $\cW$ spans $S$), $S$ is said to be equal to 
the \emph{internal direct sum} of $\cW$ (in symbols, $S = \bigoplus \cW$), and $\cW$ is said to be a \emph{maximal} linearly independent family. We write $\sum\cW$ or $\sum_{i \in I}W_{i}$ for the span of $\cW$. 
In this section and the next we generalize the results above by rephrasing them in terms of internal direct sums. 

\begin{remark}\label{conrem} Both Hamel bases and discontinuous homomorphisms are related to internal direct sums. 
\begin{itemize}
\item If $B$ is a Hamel basis for a vector space $\cS$ (over $\bbQ$), then $S$ is the internal direct sum of the one-dimensional spaces $\{ bq : q \in \bbQ\}$ ($b \in B$). 

\item If a topological vector space $S$ (again over $\bbQ$) is the internal direct sum of $\{ W_{i} : i \in I\}$ with $|I| \geq 2$, then one gets a discontinuous homomorphism from $S$ to $S$ by (for instance) sending one of these two classes to $0$ and the members of the other classes to themselves. 
\end{itemize}
\end{remark} 



We will be concerned only with the case where $S$ is the vector space $\bbR$ over $\bbQ$. 

Internal direct sums can be categorized by the dimensions of the corresponding subspaces. We write $\Phi_{\rmc}(\bbR)$ for the statement that $\bbR$ is equal to an internal direct sum of countable-dimensional spaces (this notation will be generalized in Section \ref{quessec}). 
The following generalizes one version of Theorem \ref{mainthrm}. 

\begin{theorem}\label{mainthrm2a}
If $\ZFC$ is consistent then so is $\ZF$ + $\neg \Phi_{\rmc}(\bbR)$ + ``$\bbR$ is equal to an internal direct sum of an uncountable family of subspaces."
\end{theorem}

Theorem \ref{mainthrm2} below is a more explicit version of Theorem \ref{mainthrm2a}. We the following notion of extension for linearly independent families. 

\begin{definition}\label{extensiondef} Given two families $\cW$ and $\cU$ of subspaces of some vector space $S$, we say that $\cU$ \emph{extends} $\cW$ if $\sum\cU \subseteq \sum\cW$ and for each $W \in \cW$ there is a $U \in \cU$ such that
\[U \cap \sum\cW = W.\]  
\end{definition} 

If $\cU$ extends $\cW$ and $\cU$ is linearly independent, then so is $\cW$, and there is a unique $U$ for each $W$ as Definition \ref{extensiondef}. It may be the case, however, that $U \cap \sum\cW = \{0\}$ for multiple $U \in \cU$. 

\begin{definition} Given a family of subspaces $\cW$ of some vector space $S$, and a subspace $U$, the \emph{restriction} of $\cW$ to $U$ is $\{ W \cap U : W\in \cW\}\setminus \{0\}$. 
\end{definition} 


It is straightforward to verify that extension is a transitive relation on linearly independent families, and that a family $\cW$ extends its restriction to a subspace $U$.   


\begin{definition} Suppose that $\cU_{0}$ and $\cU_{1}$ are linearly independent extensions of $\cW$, and that \[\sum \cU_{0} \cap \sum \cU_{1} = \sum \cW.\] The \emph{natural amalgamation} of $\cU_{0}$ and $\cU_{1}$ is the collection of sets $U$ of one of the two following forms. 
\begin{itemize}
\item $U$ is in some $\cU_{i}$ and $U \cap \sum\cU_{1-i} = \{0\}$. 
\item $U = U_{1}+ U_{2}$ for some $U_{1} \in \cU_{1}$, $U_{2} \in \cU_{2}$ such that \[U_{1} \cap \sum W = U_{2} \cap \sum \cW.\] 
\end{itemize} 
\end{definition} 

\begin{remark}\label{mutgendsrem} It is straightforward to check that the natural amalgamation of $\cU_{0}$ and $\cU_{1}$ is a independent extension of both $\cU_{0}$ and $\cU_{1}$.
\end{remark} 




Remark \ref{mutgendsrem} gives the following.

\begin{itemize}
\item Suppose that $\cW = \{ W_{i} : i \in I\}$ is a linearly independent family, and $v$ is a vector not in $\sum_{i \in I}W_{i}$. Then $\cW \cup \spn\{v\}$ is a linearly independent family extending $\cW$, as is the set obtained from $\cW$ by replacing any one $W_{i}$ with $W_{i} \cup \spn\{v\}$.  
\item The Axiom of Choice implies that every linearly independent family of subspaces of a vector space extends to one spanning the space. 
\item If $V \subseteq V'$ are models of ZFC and $\bbR^{V} = \bigoplus \cW$, then there is a $U$ such that $\bbR^{V'} = \bigoplus(\cW \cup \{U\})$, $U = \{ax : a \in \bbR^{V}, x \in U\}$ and $U \cap \bbR^{V} = \{0\}$. 
\end{itemize}

The following is a version of the second part of Lemma \ref{keylemma}, adapted to internal direct sums. 

\begin{lemma}\label{keylemma2} 
Suppose that $x_{1}$ and $x_{2}$ are mutually $P$-generic reals over $V$, and let $x_{3} = x_{1}-x_{2}$. Suppose that, 
\begin{itemize}
\item in $V$, $\bbR = \oplus\cW$;
\item for each $i \in \{1,2,3\}$, $U_{i}$ is, in $V[x_{i}]$ a subspace of $\bbR^{V[x_{i}]}$ such that $x_{i} \in U_{i}$, $U_{i} = \{ ax : a \in \bbR^{V},\, x \in U_{i}\}$, $U_{i} \cap \bbR^{V} = \{0\}$ and $\bbR^{V[x_{i}]} = \bigoplus(\cW \cup \{ U_{i}\})$. 
\end{itemize}  
Then, letting $U$ be the span of $\bigcup_{i=1}^{3}U_{i}$, $\cW \cup \{U\}$ is linearly independent and extends $\cW \cup \{U_{i}\}$ for each $i$ in $\{1,2,3\}$. 
\end{lemma}

\begin{proof}
We need to see that $\cW \cup \{U\}$ is linearly independent, and that the intersection of $U$ with each $\bbR^{V[x_{i}]}$ is the corresponding $U_{i}$. The second of these follows easily from the first. 

To see that $\cW \cup \{U\}$ is linearly independent, suppose that $c \in \bbR^{V}$ and $y_{i} \in U_{i}$ ($i \in \{1,2,3\}$) are such that $c + y_{1} + y_{2} + y_{3} = 0$. The first part of Lemma \ref{keylemma} then gives that $y_{3} = a(x_{1}-x_{2}) + b$, for some $a, b \in \bbR^{V}$. Since $b \in U_{3} \cap \bbR^{V}$, $b = 0$. Then $c + (y_{1} + ax_{1}) + (y_{2} - ax_{2}) = 0$. Since $\cW \cup \{U_{1}, U_{2}\}$ is linearly independent, this implies that $c$, $y + ax_{1}$ and $y_{2} - ax_{2}$ are all $0$. From this it follows that $y_{1} + y_{2} + y_{3}$ is $0$.   
\end{proof}

\section{The second model}\label{2ndmodelsec}

The second model is like the first, except that we add a family of linearly independent subspaces instead of a generic homomorphism. We let $\kappa$ and $P_{\kappa}$ be exactly as before. For the current proof $Q$ is the following partial order. Conditions are pairs $(d, \dot{\cW})$ such that $d$ is a countable subset of $\kappa$ and $\dot{\cW}$ is a $P_{d}$-name for a linearly independent set of subspaces of $\bbR$ whose span is the set of reals of the forcing extension. The order is: 
$(d_1, \dot{\cW}_1) \leq (d_0, \dot{\cW}_0)$ if $d_0 \subseteq d_1$ and $1_{P_{d_1}}$ forces that $\dot{\cW}_{1,G_{d_1}}$ extends $\dot{\cW}_{0,G_{d_0}}$. Given $q \in Q$ we write $d_{q}$ and $\dot{\cW}_{q}$ for the first and second coordinates of $q$, respectively. 

We let $\utilde{\dot{W}}$ be the natural $Q$-name for the $P_{\kappa}$-name for the unique set of linearly independent subspaces of $\bbR$ extending the realization $\dot{\cW}$ for each $(d,\dot{\cW})$ in the $Q$-generic filter $K$.
The two following lemmas show that, in the $Q * P_{\kappa}$-extension, $\bbR$ is the internal direct sum of the realization of $\utilde{\dot{\cW}}$.
Both lemmas follow from the remarks in the previous section.

\begin{lemma}\label{ohfour} For each $\alpha \in \kappa$, the set of $q \in Q$ with $\alpha \in d_{q}$ is dense. 
\end{lemma}

\begin{lemma}\label{ohfive} Every descending $\omega$-sequence in $Q$ has a lower bound. 
\end{lemma} 

Lemmas \ref{ohfour} and \ref{ohfive} give the following facts, which are unchanged from the first forcing construction of this paper. 
\begin{itemize}
\item The partial order $P_{\kappa}$ is the same in $V$ and in any $Q$-extension $V[K]$, from which it follows that the forcing iteration $Q * P_{\kappa}$ and the product forcing $Q \times P_{\kappa}$ are forcing-equivalent. 

\item Densely many conditions $(q, p) \in Q \times P_{\kappa}$ have the property that $d_{q} \subseteq \supp(p)$. We will call such conditions \emph{normal}. 


\item Letting $(K, G_{\kappa})$ denote a $V$-generic filter for $Q*P_{\kappa}$, every every element of $\bbR^{V[K, G_{\kappa}]}$ is an element of $\bbR^{V[G_{d}]}$, for some countable $d \subseteq \kappa$. In particular, $\bbR^{V[K, G_{\kappa}]} = \bbR^{V[G_{\kappa}]}$. 
\end{itemize} 





A linearly independent collection $\cW$ of subspaces of $\bbR$ induces the equivalence relation $E$ of being in the same member of $\cW$, where the domain of $E$ is the set of nonzero vectors in $\bigcup\cW$. In this way the name $\utilde{\dot{\cW}}$ induces a $Q \times P_{\kappa}$-name $\utilde{\dot{E}}$ for the corresponding equivalence relation. 
Since the equivalence relation corresponding to a linearly independent family is a subset of $\bbR^{V[G]} \times \bbR^{V[G]}$ (as opposed to a set of subsets of $\bbR^{V[G]}$) it is easier to work with equivalence relations for some of the arguments below, and in particular for defining the model $W$.

The rest of this section is a proof of the following theorem. 

\begin{theorem}\label{mainthrm2} 
Suppose that $(K, G_{\kappa})$ is a $V$-generic filter for $Q * P_{\kappa}$, and let $W$ be the model \[\HOD_{V, \bbR^{V[G_{\kappa}]}, \utilde{\dot{E}}_{K, G_{\kappa}}}.\] Then $\DC$ holds in $W$, and $(\bbR, +)^{V[K, G_{\kappa}]}$ is the internal direct sum of $\utilde{\dot{\cW}}_{K, G_{\kappa}}$. 
Furthermore, in $W$, $\bbR^{V[K,G_{\kappa}]}$ is not equal to an internal direct sum of countable dimensional subspaces. 
\end{theorem} 
 
As before, $W$ satisfies $\DC$ because it is an inner model of the model $V[K, G_{\kappa}]$, which satisfies $\ZFC$ and has the same real numbers, and every element of $W$ is ordinal definable in $V[K, G_{\kappa}]$ from a real number, a member of $V$ and the realization of $\utilde{\dot{E}}$. 

The following amalgamation lemma follows from Remark \ref{mutgendsrem}.



\begin{lemma}\label{ohsix} Suppose that $(q_{1}, p_{1})$ and $(q_{2}, p_{2})$ are conditions in $Q \times P_{\kappa}$. Suppose that $q \in Q$ is weaker than both $q_{1}$ and $q_{2}$, with $d_{q} = d_{q_{1}} \cap d_{q_{2}}$, and 
that $B^{p_{1}}_{\alpha} \cap B^{p_{2}}_{\alpha} \in I^{+}$ for all $\alpha \in \supp(p_{1}) \cap \supp(p_{2})$. Then $(q_{1}, p_{1})$ and $(q_{2}, p_{2})$ are compatible. 
\end{lemma}

The following remark carries over verbatim from our first construction. 

\begin{remark}\label{autorem2} If $a$ is a subset of $\kappa$, and $f \colon a \to \kappa$ is injective, then
$f$ induces a function $f_{Q}$ on the set of $q \in Q$ with $d_{q} \subseteq a$ and an isomorphism $f_{P} \colon P_{a} \to P_{f[a]}$ such that for each $q \in \dom(f_{Q})$ and $p \in P_{a}$, 
\begin{itemize}
\item $\supp(f_{P}(p)) = f[\supp(p)]$;
\item $B^{f_{P}(p)}_{f(\alpha)} = B^{p}_{\alpha}$ for all $\alpha \in \supp(p)$;
\item $d_{f_{Q}(q)} = f[d_{q}]$;
\item for any $V$-generic filter $G \subseteq P_{q_{d}}$, $f_{P}[G]$ is a $V$-generic filter for $P_{f[d_{q}]}$, and $\dot{\cW}_{q, G} = \dot{\cW}_{f_{Q}(q), f_{P}[G]}$.
\end{itemize}
If $a = \kappa$ and $(K, G_{\kappa})$ is $V$-generic for $Q*P_{\kappa}$, then $(f_{Q}[K],f_{P}[G_{\kappa}])$ is also $V$-generic, and the two generic extensions give rise to the same model $W$. 
\end{remark} 







Suppose toward a contradiction that $\utilde{\dot{\cU}}$ is a $Q$-name for a $P_{\kappa}$-name for a linearly independent set of countable-dimensional subspaces of $(\bbR, +)$ whose span is $\bbR$,   
as forced by some $Q\times P_{\kappa}$-condition $(q_{0}, p_{0})$.
Let $\utilde{\dot{F}}$ be a $Q * P_{\kappa}$-name for the equivalence relation corresponding to $\utilde{\dot{\cU}}$.

Replacing $(q_{0}, p_{0})$ with a stronger condition if necessary, we may assume that it is normal, and that there exist
a $v \in V$, a formula $\varphi$, a cardinal $\lambda$ and $P_{d_{q_{0}}}$-names $\dot{r}_{1},\ldots,\dot{r}_{k}$ for elements of $\bbR$ such that 
\[(q_{0},p_{0}) \forces_{Q \times P_{\kappa}} \utilde{\dot{F}} = \{ y \in \bbR^{2} : V_{\lambda}[K,G_{\kappa}] \models \varphi(y, v, \utilde{\dot{E}}_{K, G_{\kappa}}, \dot{r}_{1,G_{d_{q_{0}}}},\ldots,\dot{r}_{k,G_{d_{q_{0}}}})\}.\] 
We will write $\tau \in \utilde{\dot{F}}$ for $V_{\lambda}[K,G_{\kappa}] \models \varphi(\tau_{G_{\kappa}}, v, \utilde{\dot{E}}_{K, G_{\kappa}}, \dot{r}_{1,G_{d_{q_{0}}}},\ldots,\dot{r}_{k,G_{d_{q_{0}}}})$.

As before, the first key point is that, for each $d \subseteq \kappa$, the restriction of $\utilde{\dot{F}}_{K, G_{\kappa}}$ is decided by the restrictions of $K$ and $G_{\kappa}$ to $d$. 

\begin{lemma} 
Each normal $(q,p) \leq (q_{0}, p_{0})$ forces that for each $P_{d_{q}}$-name $\tau$ for a real number in the domain of $\utilde{\dot{E}}$, the $\utilde{\dot{F}}$-equivalence class of $\tau$ will be contained in the $P_{d_{q}}$-extension of $V$. Moreover, for each such $\tau$ there is a $P_{d_{q}}$-name $\sigma$ for a finite subset of the domain of $\utilde{\dot{F}}$ such that $p_{0}$ forces that the realization of $\tau$ will be a linear combination of the members of $\sigma$.
\end{lemma}

\begin{proof} This follows from Lemma \ref{ohsix} and Remark \ref{autorem2}. Suppose that the lemma failed for some $(q,p)$ and $\tau$. Then there exist a condition $(q',p') \leq (p,q)$ and $P_{d_{q'}}$-name $\sigma$ for a real number which is forced by $(q',p')$ to be $\utilde{\dot{F}}$-equivalent to $\tau$ and not a member of the $P_{d_{q}}$-extension of $V$. Given any $(q'',p'') \leq (q',p')$ and any $\alpha < \omega_{1}$, letting $f \colon d_{q'} \to \omega_{1}$ be an injection which is the identity function on $d_{q}$ and which maps the rest of $d_{q'}$ to elements of $\omega_{1} \setminus \alpha \cup d_{q''}$, we get by genericity that there is no countable $d \subseteq \kappa$ such that the $\utilde{\dot{F}}$-equivalence class of $\tau$ will be contained in the $P_{d}$-extension. 

The second part of the lemma is proved in the same way as the second part of the proof of Lemma \ref{locallem1}. 
\end{proof}  




We can then let, for each $q \leq q_{0}$, $\dot{F}^{q}$ be the $P_{d_{q}}$-name for the value of $\utilde{\dot{F}}_{K, G_{\kappa}} \restrict \bbR^{V[G_{d_{q}}]}$, for any $Q\times P_{\kappa}$-generic filter $(K, G_{\kappa})$ containing $(q_{0}, p_{0})$. Then $p$ forces in $P_{d_{q}}$ that the realization of $\dot{F}^{q}$ will be the equivalence relation corresponding to a decomposition of $\bbR$ into countable-dimensional subspaces in the $P_{d_{q}}$-extension. 

We have the following version of Lemma \ref{rephraselem}, which has essentially the same proof. 

\begin{lemma}\label{rephraselem} If $q,q' \leq q_{1}$, 
and $G_{d_{q}}$
and $G_{d_{q'}}$ are (respectively), $V$-generic filters for 
$P_{d_{q}}$ and 
$P_{d_{q'}}$ with $p_{1} \in G_{d_{q}} \cap G_{d_{q'}}$
such that \[\bbR^{V[G_{d_{q}}]} = \bbR^{V[G_{d_{q'}}]}\] and 
$\dot{\cW}_{q, G_{d_{q}}} = \dot{\cW}_{q', G_{d_{q'}}}$, then $\dot{F}^{q}_{G_{d_{q}}} = \dot{F}^{q'}_{G_{d_{q'}}}$. 
\end{lemma}

We now derive a contradiction, as before. 
Fix two ordinals, $\alpha_{1} < \alpha_{2}$ in $\kappa \setminus d_{q_{0}}$ and let 
$d = d_{q_{0}} \cup \{\alpha_{1}, \alpha_{2}\}$. 
We will find 
generic filters $G_{d}$ and $G'_{d}$ for $P_{d}$, such that
\begin{enumerate}
\item $p_{0}\in G_{d} \restrict P_{d_{q_{0}}} = G'_{d} \restrict P_{d_{q_{0}}}$,
\item $G_{d} \restrict P_{\{\alpha_{1}\}} = G'_{d} \restrict P_{\{\alpha_{1}\}}$ and 
\item $V[G_{d}] = V[G'_{d}]$,
\end{enumerate}
and a maximal linearly independent family of subspaces $\cW \in V[G_{d}]$ of $\bbR^{V[G_{d}]}$ into subspaces, extending $\dot{\cW}_{q_{0},G_{d} \restrict P_{q_{d_{0}}}}$, whose restriction to each of $V[G_{d_{q_{0}}\cup \{\alpha_{1}\}}]$, $V[G_{d_{q_{0}}\cup \{\alpha_{2}\}}]$ and $V[G'_{d_{q_{0}}\cup \{\alpha_{2}\}}]$ is a maximal linearly independent family in that model. We will have furthermore that $\dot{F}^{q}_{G_{d_{q}}} \neq \dot{F}^{q'}_{G_{d_{q'}}}$.
This will finish the proof.



Let $G_{d_{q_{0}}}$ be a $V$-generic filter for $P_{d_{q_{0}}}$ containing $p_{0}$. 
Let $x_{1}$ and $x_{2}$ be mutually $P$-generic reals over $V[G_{d_{q_{0}}}]$, and let $x_{3} = x_{1}-x_{2}$. Let $G_{d}$ be the $P_{d}$-generic filter extending $G_{d_{q_{0}}}$ giving rise to $x_{1}$ in coordinate $\alpha_{1}$ and $x_{2}$ in coordinate $\alpha_{2}$.  Let $G'_{d}$ be the $P_{d}$-generic filter extending $G_{d_{q_{0}}}$ giving rise to $x_{1}$ in coordinate $\alpha_{1}$ and $x_{3}$ in coordinate $\alpha_{2}$. Then $G_{d}$ and $G'_{d}$ satisfy conditions (1)-(3) above.

For each $i \in \{1,2,3\}$ let 
$W_{i}$ be, in $V[G_{d_{q_{0}}}][x_{i}]$, a subspace of $\bbR^{V[G_{d_{q_{0}}}][x_{i}]}$ such that
\begin{itemize}
\item $x_{i} \in W_{i}$, 
\item $W_{i} = \{ ax : a \in \bbR^{V[G_{d_{q_{0}}}]},\, x \in W_{i}\}$, 
\item $W_{i} \cap \bbR^{V[G_{d_{q_{0}}}]} = \{0\}$ and 
\item $\bbR^{V[G_{d_{q_{0}}}][x_{i}]}$ is the internal direct sum of $\dot{\cW}_{q_{0},G_{d} \restrict P_{q_{d_{0}}}} \cup \{ W_{i}\}$. 
\end{itemize} 
By Lemma \ref{keylemma2}, letting $W$ be the span of $\bigcup_{i=1}^{3}W_{i}$, $\dot{\cW}_{q_{0},G_{d} \restrict P_{q_{d_{0}}}} \cup \{W\}$ extends $\dot{\cW}_{q_{0},G_{d} \restrict P_{q_{d_{0}}}} \cup \{W_{i}\}$ for each $i$ in $\{1,2,3\}$, as desired.  


If $\dot{F}^{q}_{G_{d_{q}}} = \dot{F}^{q'}_{G_{d_{q'}}}$, then this set induces a set $\cU$ such that
\begin{itemize}
\item  in $V[G_{d_{q}}]$, $\cU$ is a maximal linearly independent family, 
\item the intersection of $\cU$ with each of the models $V[G_{d_{q_{0}}\cup \{\alpha_{1}\}}]$, $V[G_{d_{q_{0}}\cup \{\alpha_{2}\}}]$ and $V[G'_{d_{q_{0}}\cup \{\alpha_{2}\}}]$ is a maximal linearly independent family in this model and
\item for each $U \in \cU$ and each of the models $V[G_{d_{q_{0}}\cup \{\alpha_{1}\}}]$, $V[G_{d_{q_{0}}\cup \{\alpha_{2}\}}]$ and $V[G'_{d_{q_{0}}\cup \{\alpha_{2}\}}]$ if $U$ has a nonzero member in the model then it is contained in the model.
\end{itemize} 
If a $\cU$ as above did exist, then each of $x_{1}$, $x_{2}$ and $x_{3}$ would be a sum of elements of distinct spaces from $\cU$ in a unique way.  The equation $x_{3} + x_{2} = x_{1}$ however then gives two different sums for $x_{1}$. 

\section{Questions}\label{quessec}
 
The results above induce the following questions (and many other natural variations) which we do not know the answers to. 



We write $\Phi_{n}(\cS)$ for the assertion that a vector space $\cS$ is an internal direct sum of $n$-dimensional subspaces, and $\Phi_{\rmf}(S)$ for the assertion that $\cS$ is the internal direct sum of finite-dimensional subspaces.

\begin{remark}\label{1dimrem} Strengthening the first part of Remark \ref{conrem} above, we note that the existence of a Hamel basis for $\bbR$ is equivalent to $\bbR$ being an internal direct sum of one-dimensional subspaces (over $\bbQ$). The forward direction of this is the first part of the remark. For the reverse direction, the existence of such an internal direct sum gives a discontinuous homomorphism from $\bbR$ to itself, as noted above, and therefore an $\bbE_{0}$-selector (by \cite{LZ19}). Since the nonzero rational numbers form an abelian group under multiplication, this gives a selector for associated equivalence relation (i.e., a choice of a nonzero vector from each space), by \cite{GaoJackson}, which is then a Hamel basis.
\end{remark}


\begin{question} For which positive integers $n$ and $m$ does $\Phi_{n}(\bbR)$ imply $\Phi_{m}(\bbR)$? 
\end{question}

\begin{question} Does $\Phi_{\rmc}(\bbR)$ imply $\Phi_{\rmf}(\bbR)$ or does $\Phi_{\rmf}(\bbR)$ imply $\Phi_{n}(\bbR)$ for some or any positive integer $n$? 
\end{question}


\begin{question} Does the existence of a nontrivial homomorphism from $\bbR$ to itself imply that $\bbR$ is equal to an internal direct sum consisting of at least two subspaces? 
\end{question}

\noindent Department of Mathematics, Miami University, Oxford, OH 45056 USA. Email: \url{larsonpb@miamioh.edu}

\bls

\noindent Einstein Institute of Mathematics, The Hebrew University of
Jerusalem, Jerusalem, 91904, Israel, and Department of Mathematics, Rutgers University, New Brunswick, NJ 08854, USA.
Email: \url{shelah@math.huji.ac.il}
  

\begin{thebibliography}{99}

\bibitem{BJ95}
T. Bartoszy\'nski,  H.~I. Judah, {\bf Set theory}, A K Peters, Wellesley, MA, 1995






\bibitem{GaoJackson}
S. Gao, S. Jackson, \emph{Countable abelian group actions and hyperfinite equivalence relations}, Invent. Math. (2015)

\bibitem{H}
C. Henderson, \emph{A short proof that additive, measurable functions are linear}, \url{https://drive.google.com/file/d/1rkle78oIHH7Ypoww5N9heDd2BRh00Oda/view}

\bibitem{HS21}
H. Horowitz, S. Shelah, \emph{Transcendence bases, well-orderings of the reals and the axiom of choice}, Proc. Amer. Math. Soc. {\bf 149} (2021), no.~2, 851--858 

\bibitem{Kechris}
A.S. Kechris, {\bf Classical Descriptive Set Theory}, Springer-Verlag 1995 

\bibitem{LZ19}
P.B. Larson, J. Zapletal, \emph{Discontinuous homomorphisms, selectors and automorphisms of the complex field}, Proc. Amer. Math. Soc. 147 (2019) 4, 1733-1737

\bibitem{GST}
P.~B. Larson and J. Zapletal, {\bf Geometric set theory}, Mathematical Surveys and Monographs, 248, Amer. Math. Soc., Providence, RI, 2020

\bibitem{Pettis}
B.~J. Pettis, \emph{On continuity and openness of homomorphisms in topological groups}, Ann. of Math. (2) {\bf 52} (1950), 293--308

\bibitem{S85} 
S. Shelah, \emph{On measure and category}, Israel J. Math. {\bf 52} (1985), no.~1-2, 110--114


\end{thebibliography}
\end{document}